\documentclass[12pt]{article}
\usepackage[english]{babel}
\usepackage{geometry,amssymb}
\usepackage{amsmath}
\usepackage{amsthm}
\usepackage{bbm}
\usepackage{centernot}
\usepackage{mathtools}
\usepackage{ stmaryrd }
\usepackage{verbatim}

\newtheorem{definition}{\bf Definition}[section]

\newtheorem{theorem}[definition]{\bf Theorem}
\newtheorem{lemma}[definition]{\bf Lemma}

\newtheorem{proposition}[definition]{\bf Proposition}
\newtheorem{corollary}[definition]{\bf Corollary}
\newtheorem{remark}[definition]{\bf Remark}

\begin{document}

\title{On the fundamental groups of solenoid complements in $\mathbb{S}^3$}

\author{Xueming Hui}
\maketitle

\begin{abstract}
We show that fundamental groups of the complements of knotted solenoids in $\mathbb{S}^3$ is solely determined by a canonical sequence of knot groups. Moreover it its determined by the embedding up to mirror reflection.
\end{abstract}


\section{Introduction}
A knot is by definition an isotopy class of embeddings of $\mathbb{S}^1$ into $\mathbb{S}^3$. The main purpose of knot theory in the beginning is to tell knots apart. For example, the (right-handed) trefoil knot is different from the  unknot as shown below. There are many invariants for knots. For example, the coloring invariants, racks and quandle, fundamental groups of the knot complements, Alexander polynomials, Jones polynomials and HOMFLY polynomials etc.

\begin{figure}[ht]
	\centering
	\includegraphics[width=0.25\textwidth]{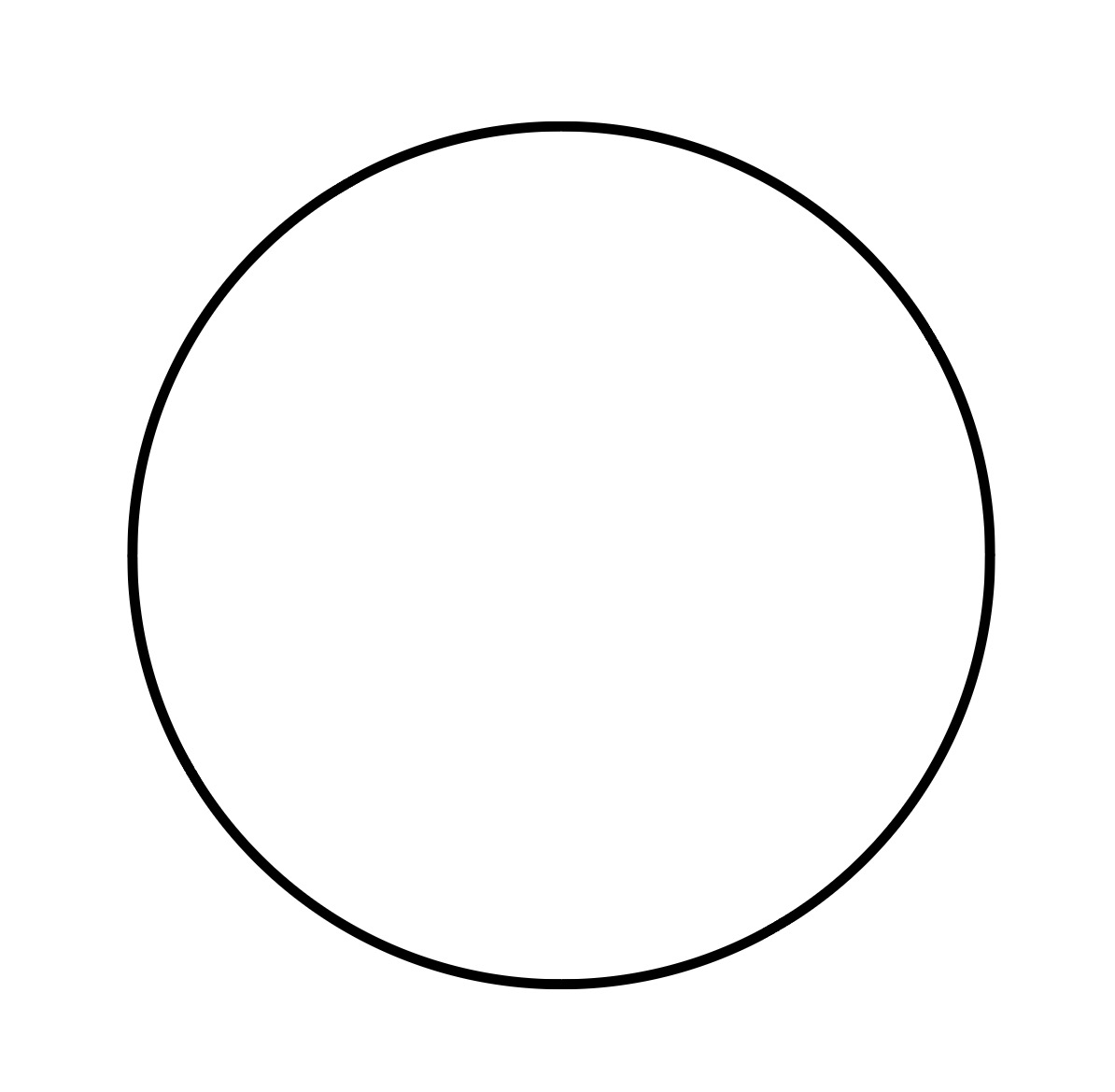}
	\includegraphics[width=0.25\textwidth]{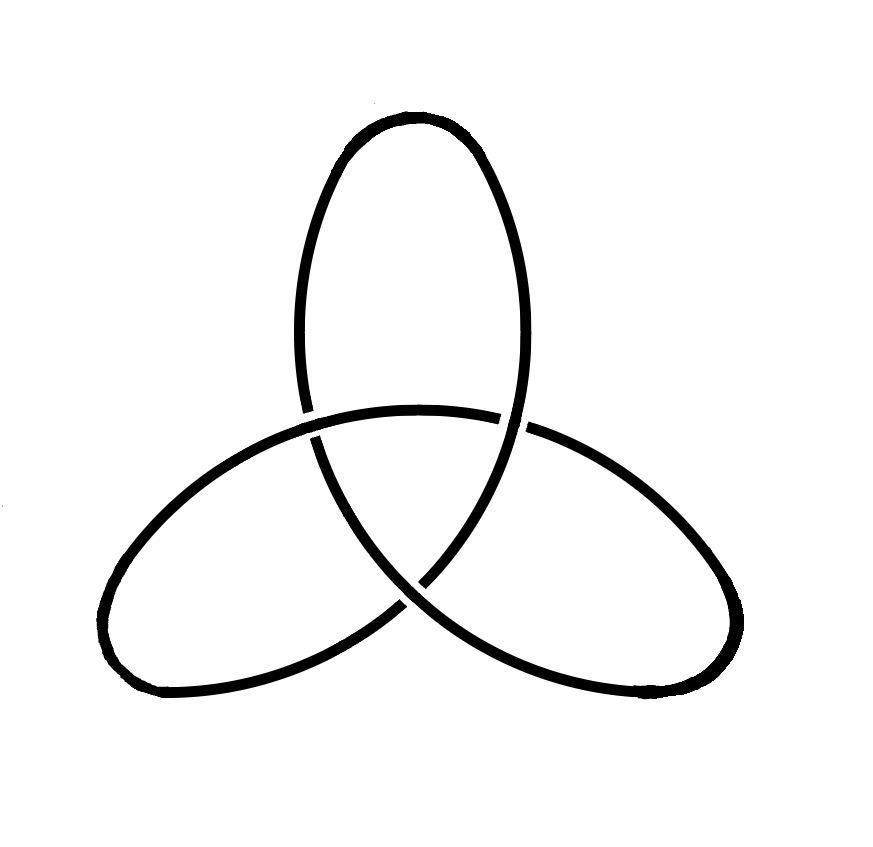}
	\caption{The Unknot and the (right-handed) trefoil knot.}
\end{figure}

In this paper we consider special sequences of knots, more specifically we consider knots sequence $\{ K_n\}_{n=0}^{\infty}$ such that $K_{n+1}$ is obtained from $K_n$ by satellite constructions. We will restrict to the case of closed braid patterns. In this case, the sequence gives an embedding of a topological object called solenoid. A \textit{Solenoid} is a topological space that is the inverse limit of an inverse system of topological groups and continuous homomorphisms 

\[ (S_i, f_i), f_i: S_{i+1}\mapsto S_i. i\geq 0\]
where each $S_i$ is a circle and $f_i$ is the map that uniformly wraps the circle $S_{i+1}$ $n_i $ times around the circle $S_i$, $n_i>1$. In other words, if we regard $S_i$'s as unit circle in the complex plane, then $f_i(z)=z^{n_i}$. Solenoid was first introduced by L. Vietoris\cite{Vietoris} in the case when $n_i=2$ and D. van Dantzig\cite{DantZig} for $n_i=n$ fixed. The general case where $n_i$ is non constant was studied by R.H. Bing, he gives a complete classification of the solenoids in \cite{Bing}. See also M.C. McCord\cite{McCord}. A 2-solenoid embedded in $\mathbb{S}^3$ is shown below. 

\begin{figure}[ht]
	\centering
	\includegraphics[width=.75\textwidth]{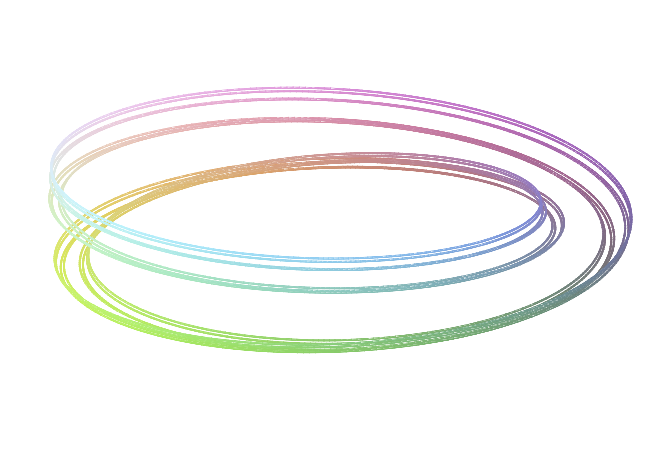}
	\caption{A Soleinoid embedded in $\mathbb{S}^3$(picture is from \cite{CMR}).}
\end{figure}
\

For an embedding of a solenoid $\Sigma$ in $\mathbb{S}^3$,  we can study the complement of the solenoid, that is $\mathbb{S}^3-\Sigma$. Like in knot theory, we would like to use the fundamental groups of the complement to study different embeddings of solenoids. Previously, B. Jiang, S. Wang, H. Zheng and Q. Zhou \cite{TAME} studied the embeddings of solenoid in $\mathbb{S}^3$. Our paper concerns the fundamental group which was considered by G. Conner, M. Meilstrup and Du{\v s}an D. Repov{\v s} \cite{CMR}. We extend their results and answer some conjectures. 

For simplicity, an isotopy class of tame solenoid embeddings is called a \textbf{soleknot}. We will discuss the precise meaning of \textit{tame} later in next section. 
\

There is a canonical sequence of knot groups and homomorphisms for a given soleknot $\Sigma$, 

\[ K_0 \xmapsto{\varphi_0} K_1 \xmapsto{\varphi_1} K_2 \xmapsto{\varphi_2} \dots \xmapsto{\varphi_{n-1}} K_n \xmapsto{\varphi_n} \dots  \] 
where all the $\varphi_n$'s are naturally induced by inclusion of knot complements. We call this sequence the \textbf{filtration} of the soleknot $\Sigma$. We have the following which is Theorem \ref{main1} in section 4. 

\begin{theorem}\label{main}
	Let $\Sigma$ and $\Sigma^{\prime}$ be two knotted soleknots, $(K_n, \varphi_n)$ and $(K_n', \psi_n)$ be the filtrations of $\pi_1(\mathbb{S}^3-\Sigma)$ and $\pi_1(\mathbb{S}^3-\Sigma^{\prime})$ respectively. Then $\pi_1(\mathbb{S}^3-\Sigma) \simeq \pi_1(\mathbb{S}^3-\Sigma^{\prime})$ if and only if $K_n \simeq K_n' $ for each $n\geq 0$. 
\end{theorem}

We will define the precise meaning of knotted soleknots, filtrations, etc later. An important fact needed for the above theorem is the following(Theorem \ref{primeSat} in section 3)

\begin{theorem}\label{primeSat1}
	A satellite knot with a closed braid pattern of winding number greater than $1$ is prime. 
\end{theorem}

For prime knot, one has the following theorem of C. Gordon and J. Luecke,

\begin{theorem}\cite{Gordon&Luecke}
    If two prime knots have isomorphic groups then they are equivalent up to mirror image.
\end{theorem}
  Together with some results on the knot subgroups of a knot group by F. Gonz\'alez-Acu\~na and W. Whitten\cite{GAW}, we are able to prove Theorem \ref{main}.

C. Gordon and J. Luecke's theorem\cite{Gordon&Luecke} shows that up to mirror image, fundamental groups distinguish the prime knots. A natural question is: do the fundamental groups of the complements of soleknots tell them apart? The following theorem(Theorem \ref{complement} in section 4) answers this question in the positive. This proves conjecture 4.4 in \cite{CMR}. 
\
\begin{theorem}\label{complement1}
	Let $\Sigma$ and $\Sigma'$ be two knotted soleknots in $\mathbb{S}^3$, $\pi_1(\mathbb{S}^3-\Sigma) \simeq \pi_1(\mathbb{S}^3-\Sigma')$ if and only if they are equivalent or their mirror images are equivalent.
\end{theorem}

\section{Preliminaries}

In this section, we introduce the basics of 3-manifolds and definitions on solenoid embeddings that will be used.

Let $S$ be a connected compact surface properly embedded in a compact oriented 3-manifold $M$. A \textbf{compressing disk} $D$ is a disk embedded in $M$ such that $D \cap S=\partial D$ and the intersection is transverse. If the curve $\partial D$ does not bound a disk inside of $S$, then $D$ is called a \textbf{nontrivial} compressing disk. If $S$ has a nontrivial compressing disk, then we call $S$ a \textbf{compressible surface} in $M$. If $S$ is neither the 2-sphere nor a compressible surface, then we call the surface \textbf{incompressible}. Assume $\chi(S) \leq 0$,  we say that $S$ is \textbf{essential} if it is incompressible, $\partial$-incompressible, and not $\partial$-parallel.

Now we are ready to define the central concept used in this paper called satellite knots. A knot $K \subset \mathbb{S}^3$ is a \textbf{satellite} if its complement contains an essential torus. 
	An equivalent and more intuitive definition is the following,
	
	Let $K_2$ be a nontrivial oriented knot in $\mathbb{S}^3$ and $V$ a closed regular neighborhood of $K_2$. Let $\tilde{V}$ be an oriented unknotted closed solid torus in $\mathbb{S}^3$ and $K_1$ an oriented knot in the interior of $\tilde{V}$. A meridional disk of $\tilde{V}$ will meet $K_1$ in a finite subset. The least number of times a meridional disk of $\tilde{V}$ must meet $K_1$ is called the \textbf{wrapping number} of the pattern. Suppose that the wrapping number of the pattern is greater than zero and let $h : (\tilde{V} ,K_1) \mapsto (V,K)$ be an oriented homeomorphism of pairs. The image of $K_1$ under $h$, denoted by $K$, is a knot in $V \subset \mathbb{S}^3$ called a \textbf{satellite knot}.
	The knot $K_2$ is called a \textbf{companion knot} of $K$ and the torus $ \partial V$ is called a \textbf{companion torus}. The pair $(K_1,\tilde{V})$ is called a \textbf{pattern} of $K$. 

When the pattern of a satellite knot is a torus knot, we call the satellite knot a \textbf{cable knot}. A \textbf{torus knot} is a knot that lies on the surface of an unknotted torus in $\mathbb{S}^3$. It's clear that satellite knot construction is highly non unique since there are infinitely many ways to identify the boundary torus of the pattern and the companion torus. We will restrict to the case of untwisted satellite knot. That is the one that sends standard longitude of the boundary torus of the pattern to the standard longitude of the companion torus. Standard is determined by their embeddings in $\mathbb{S}^3$. 

The JSJ-decomposition is used in the proof of the main theorem.
\begin{theorem}[JSJ-decomposition]
	Irreducible orientable and boundary irreducible 3-manifolds have a unique (up to isotopy) minimal collection of disjointly embedded incompressible tori such that each component of the 3-manifold obtained by cutting along the tori is either (homotopically) atoroidal or Seifert-fibered.
\end{theorem}

A 3-manifold $M$ is \textbf{irreducible} if every sphere $S$ contained in the interior of $M$ bounds a ball.

All the remaining definitions in this section are from \cite{TAME}. Some of them are stated differently in \cite{TAME}. But one can show that they are equivalent. These definitions lead to an important concept called the maximal defining sequence of a soleknot. 

\begin{definition}
    Let $N$ be a solid torus and $\beta$ be a nontrivial closed braid embedded in $N$. A closed regular neighborhood of $\beta$ in $N$ is called a \textbf{thick braid} in $N$.
\end{definition}
This is essentially a braid version of satellite knot construction. Next we give an equivalent definition of solenoid. This definition is more constructive compare to the one we give earlier in the paper. In particular, this definition also defines an embedding of solenoid in either the solid torus or $\mathbb{S}^3$. 

\begin{definition}
    Let $\{N_n\}_{n=0}^{\infty}$ be a nested sequence of solid torus such that $N_n$ is embedded in $N_{n-1}$ as a thick braid for every $n\geq 1$. If the diameter of the meridian disk of $N_n$ tends to zero uniformly as $n$ goes to infinity then we call $\Sigma = \bigcap_{n=0}^{\infty} N_n$ a \textbf{solenoid}. 
    
    The embedding $\Sigma \subset N_0$ is called a \textbf{standard embedding} of $\Sigma$ in $N_0$.  
\end{definition}

We will call $\{N_n\}_{n=1}^{\infty}$ a \textbf{defining sequence} of the standard embedding $\Sigma \subset \mathbb{S}^3$. Just like there are tame knots and wild knots. There are tame embeddings and wild embeddings for solenoids as well. 

\begin{definition}
    Let $\Sigma$ be a solenoid embedded in a solid torus $N_0$. The embedding $\Sigma \subset N_0$ is called a \textbf{tame} embedding if there is a homeomorphism $f : (N_0,\Sigma) \mapsto (N_0,\Sigma^{\prime})$ for some standard embedding $\Sigma^{\prime}\subset N_0$. 
\end{definition}
    Call $\{f^{-1}(N_n^{\prime})\}_{n=1}^{\infty}$ a \textbf{defining sequence} of the embedding $\Sigma \subset N_0$, where $\{ N_n^{\prime}\}_{n=1}^{\infty}$ is a defining sequence of $\Sigma^{\prime}$.
    
\begin{definition}
    An embedding $\Sigma \subset \mathbb{S}^3$ of a solenoid is called \textbf{tame} if it can be factored as $\Sigma \subset N_0\subset \mathbb{S}^3$ in which $\Sigma \subset N_0$ is tame. 
\end{definition}
    The sequence $\{N_n\}_{n=0}^{\infty}$ is called a \textbf{defining sequence} of the embedding $\Sigma \subset \mathbb{S}^3$. 
\begin{definition}
    A tame embedding of a solenoid $\Sigma \subset \mathbb{S}^3$ with defining sequence $\{N_n\}_{n=0}^{\infty}$ is called \textbf{knotted} if some defining solid torus $N_n \subset \mathbb{S}^3$ is knotted; otherwise we call the embedding \textbf{unknotted}.
\end{definition}

For two tame solenoid embeddings in $\mathbb{S}^3$, we can talk about when they are equivalent. 

\begin{definition}
    Call two tame solenoids $\Sigma, \Sigma^{\prime}\subset \mathbb{S}^3$ \textbf{equivalent} if there is an orientation preserving homeomorphism $f : \mathbb{S}^3 \mapsto \mathbb{S}^3$ such that $f(\Sigma) = \Sigma^{\prime}$.

    An equivalence class of tame solenoid embeddings is called a \textbf{soleknot}.
\end{definition}
 
When there is no ambiguity, the image of a tame solenoid embedding will also be called a soleknot. 

\begin{definition}
    We say two defining sequences $\{N_n\}_{n=0}^{\infty}$ and $\{N_n^{\prime}\}_{n=0}^{\infty}$ of tame solenoid embeddings in $\mathbb{S}^3$ are \textbf{strongly equivalent} if there is an orientation preserving homeomorphism $f_0 : (\mathbb{S}^3, N_0) \mapsto(\mathbb{S}^3, N_0^{\prime})$ and orientation preserving homeomorphisms $f_n : (N_{n-1}, N_n) \mapsto (N^{\prime}_{n -1} , N_n^{\prime} )$ with $f_n |\partial N_{n} = f_{n-1}| \partial N_{n-1}$ for $n\geq 1$.
\end{definition}

 The following is shown in \cite{TAME}. It says that $\Sigma$ and $\Sigma'$ having strongly equivalent defining sequences implies that $\Sigma$ is equivalent to $\Sigma'$. 

\begin{proposition}\cite{TAME} \label{max}
	Up to strong equivalence, each knotted tame solenoid $\Sigma \subset \mathbb{S}^3$ has a unique maximal defining sequence $\{N_n\}, n\geq 0$ such that $N_0$ is knotted and any other defining sequence $\{N_n^{\prime} \}, n\geq 0$ with $N_0^{\prime}$ knotted is a subsequence of $\{N_n\}, n\geq 0$. 
	
	\textup{By a \textbf{maximal} defining sequence, we mean $N_n\setminus N_{n+1}$ contains no essential torus for each $n \geq 0$.}
\end{proposition}

\section{A satellite knot with braid pattern is prime}
\

In this section, we prove a fact that will be used later. That is a satellite knot with braid pattern is prime.

Consider a braid $\beta$ with $n$ strands. Let $\hat{\beta}$ be it's closure in a solid torus $W$. Let $D$ be a meridian disc of $W$, then $W-\hat{\beta}$ is the mapping torus $M_f$ of $D-\cup_{i=1}^n\{p_i\}$. For every $i$, $p_i$ is a point in the interior of $D$ and $f$ is the mapping class of $D-\cup_{i=1}^n\{p_i\}$ determined by $\beta$. Let $U_{\hat{\beta}}$ be a tubular neighborhood of $\hat{\beta}$ in $W$, then $U_{\hat{\beta}}$ intersects $D$ at $n$ open balls $B_1,B_2,\dots, B_n$. Let $D_n:=D-\cup_{i=1}^n B_i$. Choose a base point for $D_n$ as in Figure \ref{basepoint}. 

\begin{figure}[ht]
	\centering
	\includegraphics[scale=1.2]{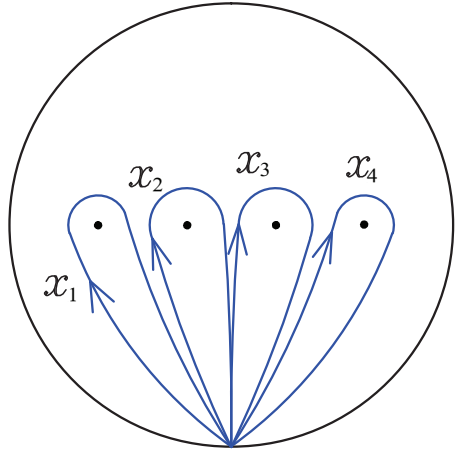}
	\caption{$D_4$ and generators of $\pi_1(D_4)$.}\label{basepoint}
\end{figure}

\begin{figure}[ht]\label{sigma1}
	\centering
	\includegraphics[scale=1.2]{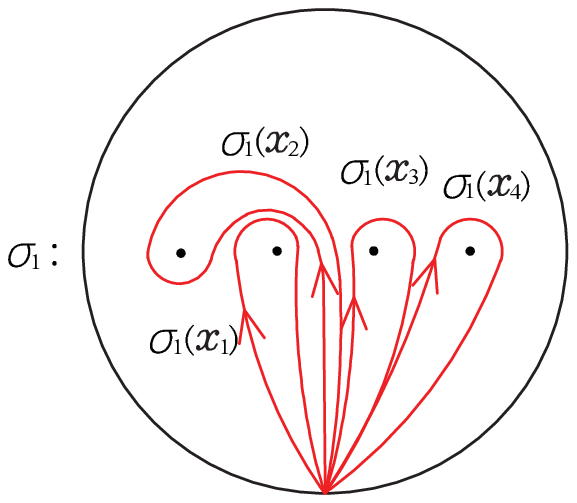}
	\caption{$\sigma_1$ acts on $\pi_1(D_4)$. } 
\end{figure}

$\beta$ acts on $\pi_1(D_n)$ as an automorphism. See Figure \ref{sigma1}. It's easy to see that $\pi_1(D_n)$ is isomorphic to the free group of rank $n$ and $\pi_1(W-U_{\hat{\beta}})\simeq \langle \ x_1,x_2, \dots ,x_n, t \ |\  t^{-1}x_i t =\beta (x_i)\ \rangle$ which is the HNN-extension of $\pi_1(D_n)$ by $\beta$. In this case, it is actually a semi-direct product. 

\begin{lemma}\label{HNN}
	Let $\beta$ be a braid with $n$ strands such $\hat{\beta}$ is a knot. Let $\pi_1(W-U_{\hat{\beta}})=\langle \ x_1,x_2, \dots ,x_n, t \ |\  t^{-1}x_i t =\beta (x_i)\ \rangle$ be a presentation of $\pi_1(W-U_{\hat{\beta}})$ as above. 
	
	The centralizer of $x_1$ in $\pi_1(W-U_{\hat{\beta}})$ is $\{ (t^{n} w)^k x_1^l \ | \ k,l \in \mathbb{Z}\}$ where $w$ is the unique element(which doesn't end with a power of $x_1$) such that $\beta^n(x_1)=wx_1w^{-1}$. 
	
\end{lemma}

\begin{proof}
	Let $y \in \pi_1(W-U_{\hat{\beta}})$ be any element that commutes with $x_1$. 
	\[ x_1 y=y x_1 \]
	\[ y^{-1}x_1y=x_1 \]
	
	As an element of $\pi_1(W-U_{\hat{\beta}})$, $y$ can be uniquely written as $y=t^mz$ for some integer $m$ and $z\in \pi_1(D_n) $. Hence,
	
	\begin{align*}
		t^{-m} x_1 t^m &=z x_1 z^{-1} \\
		\beta^m(x_1)&=z x_1 z^{-1}
	\end{align*}
	
	Braid automorphisms map $x_i$ to conjugates of $x_j$. It induces a natural permutation $\pi $ of $\{ 1,2, \dots,n\}$. The number of cycles in the cycle decomposition of $\pi$ is the number of components of the the closure of the braid. The length of each cycle is the smallest iteration it takes to map $x_i$ to a conjugate of itself. Because the braid given by $U$ is a knot, it has only one component. So the the permutation induced by the braid is an $n$-cycle and $m=kn$ for some integer $k$. If $k=1$, then $\beta^n(x_1)=wx_1w^{-1}$ for some $w$ in the free group generated by $x_1,x_2, \dots, x_n$. $w$ is unique up to multiplication by $x_1^l$ on the right. We claim that for each integer $k$, $\beta^{kn}(x_1)=t^{-kn}(t^nw)^k x_1 (t^nw)^{-k}t^{kn}$, this implies the only elements that commute with $x_1$ are $(t^{n} w)^k x_1^l, k,l \in \mathbb{Z}$. $k=0$, this is obviously true. For $k>0$, notice that
	\[ t^{-kn}(t^nw)^k x_1 (t^nw)^{-k}t^{kn}=\beta^{(k-1)n}(w)\dots \beta^n(w) w x_1 w^{-1} \beta^n(w^{-1})\dots \beta^{(k-1)n}(w^{-1})\]
	
	For $k=1$, $\beta^n(x_1)=wx_1w^{-1}$ by the definition of $w$. We assume the claim is true for $k=i$. Then 
	\begin{align*}
		\beta^{(i+1)n}(x_1)&=\beta^n(\beta^{in}(x_1))\\
		&=\beta^n(\beta^{(i-1)n}(w)\dots \beta^n(w) w x_1 w^{-1} \beta^n(w^{-1})\dots \beta^{(i-1)n}(w^{-1})) \\
		&=\beta^{in}(w)\dots \beta^{2n}(w) \beta^n(w) \beta^n(x_1) \beta^n(w^{-1}) \beta^{2n}(w^{-1})\dots \beta^{in}(w^{-1}) \\
		&=\beta^{in}(w)\dots \beta^{2n}(w) \beta^n(w) w x_1 w^{-1} \beta^{2n}(w^{-1})\dots \beta^{in}(w^{-1})
	\end{align*}
	
	Replace $k$ by $-k$ gives
	\[ t^{kn}(t^nw)^{-k} x_1 (t^nw)^k t^{-kn}=\beta^{-kn}(w^{-1})\dots \beta^{-2n}(w^{-1}) \beta^{-n}(w^{-1}) x_1 \beta^{-n}(w) \beta^{-2n}(w)\dots \beta^{-kn}(w)\]
	
	For $k>0$, \[ \beta^{kn}(x_1)=\beta^{(k-1)n}(w)\dots \beta^{2n}(w) \beta^n(w) w x_1 w^{-1} \beta^{2n}(w^{-1})\dots \beta^{(k-1)}(w^{-1}) \]
	
	Apply $\beta^{-kn}$ on both sides of the equation,
	\[ x_1= \beta^{-n}(w) \dots \beta^{-(k-2)n}(w)\beta^{-(k-1)n}(w) \beta^{-kn}(w) \beta^{-kn}(x_1) \beta^{-kn}(w^{-1}) \beta^{-(k-2)n}(w^{-1})\dots \beta^{-n}(w^{-1})\]
	
	Rearrange the equation, 
	
	\[ \beta^{-kn}(x_1) =\beta^{-kn}(w^{-1})\dots \beta^{-2n}(w^{-1}) \beta^{-n}(w^{-1}) x_1 \beta^{-n}(w) \beta^{-2n}(w)\dots \beta^{-kn}(w)\]
	This proves the claim in every case and therefore finishes the proof of the lemma. 
\end{proof}

\begin{proposition}\label{NoEssentialTorus}
	There is a unique maximal subgroup of $\pi_1(W-U_{\hat{\beta}})$ isomorphic to $\mathbb{Z}\times \mathbb{Z}$ that contains $x_1$. Furthermore, this subgroup is $i_*(\pi_1(T))$ for a torus $T$ in $W-U_{\hat{\beta}}$ that contains the base point and $i_*(\pi_1(T))$ contains $x_1$. Here $i_*$ is the homomorphism induced by inclusion. 
\end{proposition}

\begin{proof}
	The uniqueness of such a subgroup is proved in Lemma \ref{HNN}. All the elements of this subgroup are $(t^n w)^k x_1^l$, $k,l \in \mathbb{Z}$. This is the subgroup generated by $x_1$ and $t^n w$. The homomorphism onto $\mathbb{Z} \times \mathbb{Z}$ which maps $x_1$ to $(0,1)$ and $t^nw$ to $(1,0)$ is an isomorphism. Clearly, $(t^n w)^k x_1^l$ is mapped to $(k,l)$. 
	
	Let $T$ be a torus in $W-U_{\hat{\beta}}$ that contains the base point and $i_*(\pi_1(T))$ contains $x_1$. Clearly this surface exists, the torus containing the base point and parallels to $\partial U_{\hat{\beta}}$ is one such surface. 
\end{proof}

\begin{lemma}\label{IsoSurfaces}
	Let $M$ be a 3-dimensional submanifold of $\mathbb{S}^3$ with two  incompressible boundary component $T_1$ and $T_2$ that are both torus. If $i_*(\pi_1(T_1))$ is conjugate to $i_*(\pi_1(T_2))$ in $\pi_1(M)$, then $T_1$ is parallel to $T_2$. 
\end{lemma}

\begin{proof}
	Such a submanifold $M$ is either the exterior of a $2$ components link $L$ in $\mathbb{S}^3$ or it is a submanifold of a solid torus embedded in $\mathbb{S}^3$. The two cases are not exclusive. If $M$ is a submanifold of a solid torus embedded in $\mathbb{S}^3$, then since we only care about $M$ not how it is embedded in $\mathbb{S}^3$, we can construct a new embedding of $M$ such that it can be viewed as a link complement $\mathbb{S}^3-L$ where one component of $L$ is the unknot in $\mathbb{S}^3$. 
 
    Denote the components of $L$ by $K_1$ and $K_2$. Then $T_1$ and $T_2$ are the boundaries corresponds to $K_1$ and $K_2$ respectively. The meridian $m_1$ of $T_1$ can be deformed into $T_2$ since $i_*(\pi_1(T_1))$ is conjugate to $i_*(\pi_1(T_2))$. Furthermore, $m_1$ can be deformed into a simple closed curve in $T_2$, this curve is non trivial in $i_*(\pi_1(T_1))$, so it is also non-trivial in $i_*(M)$ and $i_*(\pi_1(T_2))$. Therefore it can be identified with the curve $\gamma=(p,q)$ where $p$ and $q$ are coprime. The Dehn filling of $m_1$ makes $\gamma$ trivial in the exterior of $K_2$. This is impossible if $K_2$ is knotted since the boundary is $\pi_1$-injective in this case. So $K_2$ has to be the unknot. In particular, $m_1$ is homotopic to the longitude of $K_2$. Do the same argument for $m_2$, we have $K_1$ is also the unknot. Moreover, $K_1$ intersects a disc bounded by $K_2$ at just one point. This implies that $L$ has to be the Hopf link. Hence, $M$ is homeomorphic to $T\times [0,1]$ and $T_1$ is parallel to $T_2$. 
\end{proof}

\begin{theorem}\label{NoGeoTorus}
	There is no essential torus in $W-U_{\hat{\beta}}$ such that the meridian of $U_{\hat{\beta}}$ is homotopic to a curve on it. In particular, there is no swallow-follow torus in $W-U_{\hat{\beta}}$. 
\end{theorem}
\begin{proof}
	If there is one such essential torus, then there are two different $\mathbb{Z} \times \mathbb{Z}$ subgroups of $\pi_1(W-U_{\hat{\beta}})$ containing $x_1$ since the boundary torus $\partial U_{\hat{\beta}}$ is another one. This contradicts Proposition \ref{NoEssentialTorus} and Lemma \ref{IsoSurfaces}. The second part is true because for any swallow follow torus $S$ in $W-U_{\hat{\beta}}$, the meridian of $U_{\hat{\beta}}$ is homotopic to the meridian of $S$.
\end{proof}

\begin{theorem}\label{primeSat}
	A satellite knot with a closed braid pattern of winding number greater than $1$ is prime. 
\end{theorem}

\begin{proof}
	Let $K$ be a satellite with companion torus $V$ and pattern $P \subset W$. Since the pattern $P$ is a closed braid of winding number greater than $1$, $K$ is a proper satellite. Proof by contradiction, assume there is a factorising sphere $S$ which decomposes $K$ as a product. We assume that $K$, $S$ and $\partial V$ are in general position. 
	
	The argument in the proof of Theorem 4.4.1 in Cromwell\cite{cromwell} can be adapt without any changes. Therefore, $S$ must lie inside $V$ bounding a 3-ball $B \subset V$, and its preimage $h^{-1}(S)$ in $W$ decomposes the pattern $P$ as product of two nontrivial factors. 
	
	Assign the pattern $P$ an orientation, then $P$ intersects $h^{-1}(S)$ in two points, one for $P$ entering $B$, one leaving it. Let $U$ be an open tubular neighborhood of $P$, then $T:=\partial(B-U)$ is a swallow follow torus for $P$. Since $P$ is a closed braid in $W$, this contradicts Theorem \ref{NoGeoTorus}. Hence $h(P \cap B)$ is an unknotted tangle. So $K$ must be prime. 
\end{proof}
The last step may seem intuitively obvious, but the author didn't found a satisfying geometric proof. We have another proof using JSJ-decomposition, but we think the algebraic approach is more clear and elegant in this case. Theorem \ref{primeSat} is probably known, but as the author can tell, it was not written down anywhere in the literature.  

\section{The fundamental groups of soleknots complement in $\mathbb{S}^3$}

It's easy to see that all the higher homotopy groups of the soleknot complements in $\mathbb{S}^3$ are trivial. This is the following. 

\begin{theorem}\label{K(G,1)}
	Let $\Sigma \subset \mathbb{S}^3$ be a soleknot. Then $\mathbb{S}^3-\Sigma$ is an Eilenberg-MacLane space $K(G,1)$ where $G=\pi_1(\mathbb{S}^3-\Sigma)$. 
\end{theorem}

\begin{proof}
	Let $n\geq 2$ and $f$ be any continuous map from $\mathbb{S}^n$ to $\mathbb{S}^3 -\Sigma$. By continuity of $f$ and compactness of $\mathbb{S}^n$, $f(\mathbb{S})$ is contained in a compact subset of $\mathbb{S}^3-\Sigma$. Therefore there exists a torus $T$(this is one of the defining torus of $\Sigma$) in $\mathbb{S}^3- \Sigma$ such that $f(\mathbb{S}^n)$ is contained in the compact component of $\mathbb{S}^3 -\Sigma- T$. This component is a knot complement. It's a classical theorem of knot theory that knot complements are aspherical. Hence, $f$ restricts on this component is homotopic to a constant map. Therefore it is homotopic to a constant map in $\mathbb{S}^3-\Sigma$. This shows that $\pi_n(\mathbb{S}^3-\Sigma)=0$ for every $n\geq 2$ and $\mathbb{S}^3-\Sigma$ is Eilenberg-MacLane space $K(G,1)$, where $G=\pi_1(\mathbb{S}^3-\Sigma)$. 
\end{proof}

\begin{corollary}
	Every isomorphism between soleknot groups is realized by a homotopy equivalence unique up to homotopy. 
\end{corollary}

\begin{corollary}
	Every automorphism of a soleknot group is induced by a homotopy equivalence unique up to homotopy. Every self homotopy equivalence induces an automorphism up to conjugacy. 
\end{corollary}

Before proving our next theorem, we first introduce some notations and terms. First, recall Proposition \ref{max} gives a unique maximal defining sequence for a soleknot. For a knotted soleknot $\Sigma$, this maximal defining sequence gives a canonical sequence of knot exteriors $M_n$ and knot groups $K_n$ and homomorphisms, 
\[ K_0 \xmapsto{\varphi_0} K_1 \xmapsto{\varphi_1} K_2 \xmapsto{\varphi_2} \dots \xmapsto{\varphi_{n-1}} K_n \xmapsto{\varphi_n} \dots  \] 
where all the $\varphi_n$'s are naturally induced by inclusion of knot complements. By Lemma 2.1 in \cite{CMR}, for any solid tori $T_1$ and $T_2$ in $\mathbb{S}^3$ with $T_1 \subset int(T_2)$, let $J$ be the core curve of $T_1$ and $K $ the meridian curve of $T_2$, linking number $lk(J, K)\neq 0$ implies that the homomorphism $\pi_1(\mathbb{S}^3-T_1) \mapsto \pi_1(\mathbb{S}^3-T_2)$ induced by inclusion is injective. Clearly, any thick braid of winding number greater than 1 satisfies the condition, therefore $\varphi_n$'s are all injective. And the base point for all the fundamental groups will be a chosen point in $\mathbb{S}^3-N_0$. The direct limit of this sequence is the fundamental group of $\mathbb{S}^3-\Sigma$. Call this sequence the \textbf{filtration} of $\pi_1(\mathbb{S}^3-\Sigma)$. We will also denote the filtration by  $(K_n, \varphi_n)$. 

\begin{theorem}\label{homeomorphic type and filtration}
	Let $\Sigma$ and $\Sigma'$ be two knotted soleknots in $\mathbb{S}^3$, $K_n \simeq K_n' $ for each $n\geq 0$ if and only if they are equivalent or their mirror images are equivalent.
\end{theorem}

\begin{proof}
    We adopt the notations from previous discussions. 
	The `if' part is a direct consequence of Proposition \ref{max}. 
 
    We next show the `only if' part: for any $n>0$, $K_n$ and $K_n'$ are knot groups of some prime knots by Theorem \ref{primeSat}. If $K_n  \simeq K_n$, by the classical theorem of C. Gordon and J. Luecke\cite{Gordon&Luecke}, $M_n$ is homeomorphic to $M_n'$. Moreover, up to mirror reflection, $M_n$ and $M_n'$ are knot complements of the same knot in $\mathbb{S}^3$. So up to mirror reflection, we can identify $M_n$ and $M_n'$. The uniqueness of JSJ-decomposition of $M_n$ and $M_n'$ implies that the embeddings of $M_{n-1}$ in $M_n$ and $M_{n-1}'$ in $M_n'$ are isotopic (we identified $M_n$ and $M_n'$). This implies that $\Sigma$ and $\Sigma'$ have strongly equivalent maximal defining sequences. Therefore $\Sigma$ and $\Sigma^{\prime}$ are equivalent up to mirror reflection. 
\end{proof}

The proof of the Theorem \ref{main1} relies on some tools from \cite{GAW}. In \cite{GAW}, F. Gonz\'alez-Acu\~na and W. Whitten studied knot subgroups of a knot group. By knot group we mean a group that is isomorphic to the fundamental group of the complement of a knot in $\mathbb{S}^3$. Not all theorems and definitions from \cite{GAW} will be shown here. We will only quickly go through the ones are used in the proof of Theorem \ref{main1} and explain the ideas as much as we can. The readers are encouraged to read \cite{GAW} for technical details. We will not try to explain everything in the original statement of theorems and definitions. 

 Let $M$ be a knot exterior, denote the union of the JSJ-pieces of $M$ that meet the boundaries of $M$ by $\gamma M$. A subgroup $H$ of $\pi_1(M)$ is \textbf{loose} if, for some component $C$ of $M- \gamma M$, there is a conjugate of $H$ contained in $i_*(\pi_1 (C))$, where $i: C \mapsto M$ is inclusion. Otherwise, $H$ is \textbf{tight}. 

   Let $G$ be a knot group, and let $H < G$. Then $H$ is a \textbf{companion} of $G$, if there is a knot complement $E$ containing an essential torus $T$ and if there is an isomorphism $\phi : G \mapsto \pi_1(E)$ sending $H$ onto $i_*(\pi_1 (E_1))$ where $E_1$ is the component of $cl(E-T)$ that is a knot complement and $i: E_1 \mapsto E$ is inclusion. 

\begin{remark}\cite{GAW}\label{rem}
	\textup{If $E$ is the complement of a prime knot $K$, then the complements $E_1, \dots, E_r$ of the companions of $K$(in the sense of satellite knots) are naturally embedded in $E$, and $\pi_1 (E_1), \dots, \pi_1 (E_r)$ are up to conjugacy all the loose companions of $\pi_1 (E)$. }
\end{remark}

\begin{theorem}\cite{GAW}\label{Redu}
	Any noncyclic knot-subgroup of a knot group $G$ is a tight subgroup of $G$ or of a loose companion of $G$. 
\end{theorem}

Theorem \ref{Redu} reduces the problem of finding the knot subgroups of a knot group to that of finding the tight subgroups of the knot group and its loose companions. We have a good understanding of loose companions especially in the case of prime knots by remark \ref{rem}. The following theorem classifies all the tight subgroups of a prime satellite knot. 

	\begin{theorem}\cite{GAW}\label{7.11}
		Let $G$ be the group of a prime satellite knot $K$, and let $G_1$ be the group of a nontrivial knot $K_1$(If $K_1$ is a cable knot, assume that $G_1 \ncong G $ ). Then $G_1$ properly embeds in $G$ as a tight subgroup if and only if there are integers $s, t, p, d, \epsilon,$ and $\delta$ such that 
		\begin{itemize}
			\item[\textbf{1.}] $K$ is the $(s,t)-$cable of the $(p,q)-$torus knot;
			\item[\textbf{2.}] $pq-\frac{s}{t}=-\frac{\epsilon}{d}+\frac{\delta z w}{d t}$, $d>1$, $|\epsilon|=1$, and $|\delta|\leq 1$, where $z=(t,d)>1$ and $w=(s, dpq+\epsilon)$; and
			\item[\textbf{3.}] if $|\delta|=1$, then $K_1$ is the $(sw^{-1}, \epsilon \delta z)-$cable of the $(p,q)-$torus knot, and if $\delta =0 $, then $K_1$ is a composite knot every prime factor of which is a $(p,q)-$torus knot or its mirror image. 
		\end{itemize}
	\end{theorem}

The technical details in this theorem is not important, the key point here is that $G_1$ properly embeds in $G$ as a tight subgroup if and only if $K$ is a cable knot of a torus knot. Theorem \ref{primeSat} proves that this is all we need when the pattern is a closed braid which is how we construct solenoids. 

\begin{theorem}\label{main1}
	Let $\Sigma$ and $\Sigma^{\prime}$ be two knotted soleknots, $(K_n, \varphi_n)$ and $(K_n', \psi_n)$ be the filtrations of $\pi_1(\mathbb{S}^3-\Sigma)$ and $\pi_1(\mathbb{S}^3-\Sigma^{\prime})$ respectively. Then $\pi_1(\mathbb{S}^3-\Sigma) \simeq \pi_1(\mathbb{S}^3-\Sigma^{\prime})$ if and only if $K_n \simeq K_n' $ for each $n\geq 0$. 
\end{theorem}

\begin{proof}
    We have proved in Theorem \ref{homeomorphic type and filtration} that $K_n \simeq K_n'$ for every $n$ implies that the soleknots $\Sigma$ and $\Sigma'$ are equivalent. Therefore the soleknot complements $\mathbb{S}^3-\Sigma$ and $\mathbb{S}^3-\Sigma'$ are homeomorphic. Hence, $\pi_1(\mathbb{S}^3-\Sigma) \simeq \pi_1(\mathbb{S}^3-\Sigma^{\prime})$. This proves one direction.
    
    Now let $F: \pi_1(\mathbb{S}^3-\Sigma) \mapsto \pi_1(\mathbb{S}^3-\Sigma^{\prime})$ be an isomorphism. Since all the $\psi_n$'s are injective, the canonical map $K_n' \mapsto \pi_1(\mathbb{S}^3-\Sigma^{\prime})$ is injective and $K_n'$ can be regarded as subgroups of $\pi_1(\mathbb{S}^3-\Sigma^{\prime})$. By the classical theorem of C. Gordon and J. Luecke, up to mirror image, fundamental groups distinguish the prime knots. So $\{K_n\}_{n=1}^{\infty}$ are pairwise non isomorphic since by Theorem \ref{primeSat}, these are all prime knots. For each $i\geq 1$, $K_i$ is finitely generated, $F(K_i)$ is a finitely generated subgroup of $\pi_1(\mathbb{S}^3-\Sigma^{\prime})$. So there exists a smallest natural number $n_i$ such that $F(K_i) < g^{-1}K_{n_i}'g$ for some $g\in \pi_1(\mathbb{S}^3-\Sigma^{\prime})$. By the results from \cite{GAW}, we know that any knot groups can only have finitely many knot subgroups up to isomorphisms. So we can choose an $i>0$ such that $n_i$ is greater than $1$.  
	
	 By Theorem \ref{Redu}, $F(K_i)$ is a tight subgroup of $g^{-1}K_{n_i}'g$ or of a loose companion of $g^{-1}K_{n_i}'g$. Every loose companion of $g^{-1}K_{n_i}'g$ will give rise to a knot complement of a companion knot of the knot $L$ corresponding to $M_{n_i}'$(This means $M_{n_i}'$ is the knot complement of $L$). Its boundary must be isotopic to one of the tori in the JSJ-decomposition of $M_{n_i}'$. The JSJ-decomposition of $M_{n_i}'$ is a subset of $\{M_k'\}_{k=0}^{n_i-1}$ and the JSJ-decomposition of $M_0'$. Therefore $F(K_i)$ is not a tight subgroup of a loose companion of $g^{-1}K_{n_i}'g$ since $F(K_i)$ is not contained in any conjugates of $g^{-1}K_{n_i-1}'g$ by the choice of $n_i$. So $F(K_i)$ is a tight subgroup of $g^{-1}K_{n_i}'g$. 
	 
	  By uniqueness of JSJ-decomposition, $M_{n_i}'$ is the knot complement of a cable knot of a torus knot if and only if $n_i=1$. So it is not the complement of the cable of a torus knot since $n_i>1$. By Theorem \ref{7.11}, the group of a nontrivial knot properly embeds in the group of a prime satellite knot $K$ as tight subgroup if and only if $K$ is a cable of the some torus knot. This implies that $F(K_i)$ does not properly embed in $g^{-1}K_{n_i}'g$, therefore $F(K_i)=g^{-1}K_{n_i}'g$. $F(K_i)=g^{-1}K_{n_i}'g$ implies that $K_i$ and $K_{n_i}'$ correspond to the same prime knot up to mirror image. So $M_i$ is homeomorphic to $M_{n_i}'$. By the uniqueness of JSJ-decomposition of $M_i$ and $M_{n_i}'$, $i=n_i$ and $K_n \simeq K_n'$ for every $n\leq i$. The same argument works for all $n\geq i$, hence $K_n\simeq K_n'$ for all $n\geq 0$.
\end{proof}

Combine Theorem \ref{homeomorphic type and filtration} and Theorem \ref{main1}, we have the following. 
\begin{theorem}\label{complement}
	Let $\Sigma$ and $\Sigma'$ be two knotted soleknots in $\mathbb{S}^3$, $\pi_1(\mathbb{S}^3-\Sigma) \simeq \pi_1(\mathbb{S}^3-\Sigma')$ if and only if there is a homeomorphism $f$ of $S^3$ such that $f(\Sigma)=\Sigma' $, in other words, they are equivalent or they are mirror image of each other.
\end{theorem}

Theorem \ref{complement} improves Theorem 5.4 in \cite{CMR}. They show that there exists uncountably many inequivalent knotted soleknots complements using hyperbolic structures. Notice that by \cite{CMR} and \cite{TAME}, the fundamental group of the complements of unknotted soleknots only determine the solenoids as a topological space, not the tame embeddings or soleknots as we call them here. In fact, there are uncountably many different unknotted soleknots with isomorphic fundamental group for their complements. 
\section*{Acknowledgments}
The author would like to thank Gregory Conner for introducing this problem to him and conversations on the topic, Mark Hughes and Jessica Purcell for interests and comments on the paper.

\renewcommand\refname{Reference}
\bibliographystyle{plain}
\bibliography{template_Article}

\end{document}